\documentclass[12pt]{amsart}
\usepackage[utf8]{inputenc}
\usepackage{mathtools,amssymb}
\usepackage{xcolor}
\usepackage{bm}

\usepackage[margin=0.75in,top=0.8in]{geometry}
\everymath{\displaystyle}

\newtheorem{theorem}{Theorem}[section]

\newtheorem{definition}{Definition}[section]

\newtheorem{lemma}{Lemma}[section]

\newtheorem{proposition}{Proposition}[section]
\newtheorem{remark}{Remark}[section]

\usepackage{hyperref}
\title{On exact Observability for Compactly perturbed infinite dimension systems}

\author{Nisrine Charaf\dag}
\address{ \dag Laboratoire Jean Kuntzmann,  UMR CNRS 5224, 
Universit\'e  Grenoble-Alpes, 
Bat. IMAG, 150 Pl. du Torrent 38400, St Martin d’H\`eres, France.}
\email{Nisrine.Charaf@univ-grenoble-alpes.fr}

\author{Faouzi Triki\ddag}
\address{\ddag Laboratoire Jean Kuntzmann,  UMR CNRS 5224, 
Universit\'e  Grenoble-Alpes, 
Bat. IMAG, 150 Pl. du Torrent 38400, St Martin d'H\`eres, France.}
\email{faouzi.triki@univ-grenoble-alpes.fr}

\begin{document}

%

\begin{abstract}
In this paper, we study the observability of compactly perturbed infinite dimensional systems. Assuming that a given
infinite-dimensional system with self-adjoint generator is exactly observable we derive  sufficient conditions 
on a compact self adjoint perturbation to guarantee that the perturbed system stays exactly observable. The analysis is based  on a careful asymptotic estimation of the spectral elements of the perturbed unbounded operator in terms of the compact perturbation. These intermediate results are  of importance themselves. 
\end{abstract}



\maketitle

\section{Introduction}\label{observability: section1}
\noindent 

Let $X$ be a Hilbert space with norm $\|\cdot\|_{X}$ and inner product $\langle \cdot, \cdot \rangle_X$.  
Let $A : X \to X$ be a linear, unbounded, self-adjoint, non-negative operator with compact resolvent and domain $D(A)$.  
We define the scale of Hilbert spaces $(X_{\beta})_{\beta \in \mathbb{R}}$ associated with $A$ by  
$X_{\beta} = D(A^{\beta/2})$ with norm $\|z\|_{X_{\beta}} = \|A^{\beta/2} z\|_{X}$ for $\beta \geq 0$,  
and for $\beta < 0$ we set $X_{-\beta} = X_{\beta}^*$, the dual space with respect to the pivot space $X$.

\noindent The operator $A$ can be extended or restricted to each space $X_{\beta}$ 
so that it becomes a bounded operator
\[
A : X_{\beta} \;\longrightarrow\; X_{\beta-2}, 
\qquad \forall\, \beta \in \mathbb{R}.
\]

\noindent According to Stone's theorem, $iA$ generates a strongly continuous group of isometries in $X$ denoted $(e^{itA})_{t {\in \mathbb{R}}}$ \cite{TW}.
\\
Further, let $Y$ be a  Hilbert space equipped with the norm and scalar product respectively ${\lVert \cdot \rVert_{Y}}$ and $\langle \cdot, \cdot \rangle_Y$.
Let $C: D(A) \to Y$ be a bounded linear operator.  For $z_0\in X$, and $y\in Y,$  we consider the following infinite-dimensional observation system

\begin{equation}\label{observability1}
\left\{ \begin{array}{lllcc}
   \dot {z} (t)  = iA z(t) , \quad  t>0, \\
    y(t) = Cz(t), \quad t>0, \\
    z(0) =z_0.
\end{array}\right.
\end{equation}
 The element $z_0$ is called the initial state, $z(t)$ is called the state at time $t$, and $y$ is the output function.
Note that although $z_0 \in D(A)$, the function $z(t)$ need not belong to $D(A)$. Therefore in order to be able to
define the output function $y$ a continuous extension of $C$ to the whole space $X$ is required.  

\begin{definition}
    The operator $C$ in system \eqref{observability1} is  an admissible observation operator if for every $T>0$ there exists a constant $K_{T}>0$ such that

    \begin{equation}\label{observability2}
        \int_0^T {\lVert y(t) \rVert^{2}_{Y}} dt \leq K_{T} {\lVert z_{0} \rVert^{2}_{X}}, \quad 
 \forall z_{0} \in D(A).
 \end{equation}

 If $C$ is bounded, i.e. it can be extended such that $C$ $\in$ $\mathcal{L}( X, Y)$, then $C$ is clearly  an admissible observation operator
 with $K_{T} = T\|C\|^{2}$.
 
\end{definition}

We further assume that the operator $C$ is admissible. Next we introduce the concept of exact observability. 

\begin{definition}
    System \eqref{observability1} is exactly observable in time $T$ if there exists a constant $k_{T}$ $>$ 0 such that 
     \begin{equation}\label{observability3}
       k_{T} {\lVert z_{0} \rVert^{2}_{X}} \leq    \int_0^T {\lVert y(t) \rVert^{2}_{Y}} dt ,\quad
 \forall z_{0} \in D(A).
 \end{equation}
System \eqref{observability1} is said to be exactly observable if it is exactly observable for a given time $T>0$.
\end{definition}
\noindent The observability inequality can be interpreted as a stability estimate for the inverse problem of recovering the initial state $z_0$ from the knowledge of the observation $y(t)$, $t \in (0,T)$, where $T>0$ is chosen sufficiently large. It is also known that exact observability and exact controllability are dual properties \cite{Rusell 1977}. These properties can be established using time-domain techniques such as non-harmonic Fourier series \cite{AI,KL}, the multiplier method \cite{Li}, and microlocal analysis techniques \cite{BLR}, or by frequency-domain techniques in the spirit of the well-known Fattorini-Hautus test for finite-dimensional systems \cite{AT,Fa,Ha,ZY}.\\

   \noindent In this paper, we are interested in the exact observability of weakly perturbed systems. Specifically, assuming that the system in \eqref{observability1} is exactly observable, our objective is to derive sufficient conditions for an unbounded perturbation \( K \) of \( A \) such that the system remains exactly observable when \( A \) is replaced by \( A + K \).\\

 \noindent The analysis is carried out using frequency domain techniques and resolvent estimates. Throughout we assume that the operator $A$ has a compact resolvent and therefore, that the spectrum of $A$ is formed by isolated eigenvalues. More precisely, since $A$ is self-adjoint and  positive,  the spectrum of $A$ is given by $\sigma(A) = \{\mu_{k}\}_{k \in \mathbb{N}^{*}}$,  where $(\mu _{k})_{k \in \mathbb{N}^{*}}$ is a sequence of positive increasing real numbers. Denote 
$(\phi_k)_{k\in \mathbb N}$ the normalized eigenfunctions  associated to the eigenvalues $(\mu _{k})_{k \in \mathbb{N}^{*}}$,
that is \[ A\phi_k = \mu_k \phi_k, \;\; \|\phi_k\|_X = 1, \qquad k \in \mathbb N^*.\]


 \noindent The plan of the paper is as follows. In Section~2, we recall known characterizations of exact observability in the frequency domain. The main results of the paper are presented in Section~3. More precisely, we first derive an asymptotic relation between the perturbed and unperturbed eigenvalues in Theorem~\ref{eigenvalues}. Then, under additional assumptions on the perturbation, we prove exact observability for the perturbed system in Theorem~\ref{main theorem}. Sections~4 and~5 are devoted to the proofs of Theorems~\ref{eigenvalues} and~\ref{main theorem}, respectively.
 \section{Characterization of exact Observability}

We recall the following  result derived  in \cite[Theorem 4.4]{Russel} by this theorem.
\begin{theorem} \label{thm1}
    \noindent Let $A$ be a  self-adjoint, positive with compact resolvent operator, and  let  $C$ be an  admissible operator for the system \eqref{observability1}.
    Assume  the following gap condition 
\begin{equation} \label{gap condition}
 {\mu}_{k+1} - {\mu}_{k} > \gamma, 
\end{equation} 
holds for some constant $\gamma >0$.\\

 Then, the system \eqref{observability1} is exactly observable if and only if there exists $\delta$ $>$ 0 such that for all $k$ $\in$ $\mathbb N^{*}$ 
 \begin{equation} \label{observabilitycondition}
  \left \lVert C\phi_k \right \rVert^2_{Y} \geq \delta, \qquad \forall k\in \mathbb N^*.
  \end{equation}
 \end{theorem}

 


We also recall this result (see \cite{Burq} for the proof).

\begin{theorem}
The system \eqref{observability: section1} is exactly observable if and only if there exists a constant $\rho>$ 0 such that the following inequality holds 
 \begin{eqnarray} \label{Hautus condition}
 \left \lVert (A - wI)z \right \rVert^{2}_{X} + \left \lVert Cz \right \rVert^{2}_{Y} \geq \rho \left \lVert z \right \rVert^{2}_{X},  \quad \forall \omega \in \mathbb{R}, \quad   \forall z \in D(A),
 \end{eqnarray}
 where $I$ is the identity operator. 
\end{theorem}

\section{Main results}

In this section we present the main results of the paper.  
We consider the infinite-dimensional observation system described by  
\begin{equation} \label{observability4}
\left\{
\begin{array}{ll}
   \dot{z}(t) = i A_K z(t), & t > 0, \\[0.3em]
   y(t) = C z(t), & t > 0, \\[0.3em]
   z(0) = z_0,
\end{array}\right.
\end{equation}
where $A_{K} = A + K$, and 
$K : D(A) \to X$,  is a self-adjoint, non-negative, and compact operator.  \\

\noindent The goal of the paper is to study the observability properties of the  evolution system \eqref{observability4}. Precisely, assuming that the unperturbed system \eqref{observability1} is exactly observable and that $K$ is a self-adjoint perturbation of $A$, we investigate  sufficient conditions under which the perturbed system \eqref{observability4} remains exactly observable.  \\

\noindent It is straightforward to verify that the operator $A_K$ has a compact resolvent; consequently, its spectrum consists of isolated eigenvalues. We denote by $(\tilde{\mu}_n)_{n \in \mathbb{N}^*}$ the increasing sequence of eigenvalues of the perturbed operator $A_K$, and by $(\tilde{\phi}_n)_{n \in \mathbb{N}^*}$ the associated normalized eigenfunctions. They satisfy
\[
A_K \tilde{\phi}_k = \tilde{\mu}_k \tilde{\phi}_k, \qquad  \forall k \in \mathbb{N}^*.
\]

\noindent We first assume  the following weak  necessary condition for exact observability on the eigenfunctions of the operator $A_K$: 
      
   \begin{equation}
      \label{necessarycondition}
      \|C \tilde \phi_k \|_{Y} \not= 0, \qquad \forall k\in \mathbb N^*.
  \end{equation}
  
  \medskip

\noindent This condition is also necessary for weak observability (see for instance \cite{AT} and references therein). We will see later the assumption concerns mainly the low-frequency eigenfunctions ($k \leq k_\rho$ for some integer $k_\rho$). Indeed, taking 
\[c_n = \min_{k \leq n} \| C \tilde{\phi}_k \|_{Y}, \qquad n\in \mathbb N^*,\] it is straightforward to see that the system \eqref{observability4} is exactly observable if and only if the non-increasing sequence $c_n>0$ does not converge to zero. For a fixed $k \in \mathbb N^*$ it is possible to construct a perturbation $K$ within a general class of operators that breaks the condition  \eqref{necessarycondition}.\\


\medskip


We first derive the following  relationship between the eigenvalues of  $A$ and $A+K$ that is of interest itself.

\begin{theorem}\label{eigenvalues}
There exists a function $f \in C^0(\mathbb R_+)$ satisfying   $\lim_{\mu \to +\infty}f(\mu) = 0$ such that
 \begin{equation} \label{identity0}
    \tilde \mu_n = \mu_n ( 1 + f(\mu_n)), \quad \forall n\in \mathbb N^*.
\end{equation}    

\end{theorem}

The asymptotic identity \eqref{identity0} shows that the eigenvalues of $A+K$ are small relative perturbation of those
of $A$. 
The following  is the main result of the paper.

 \begin{theorem} \label{main theorem} Assume the system \eqref{observability1} is exactly observable,
$K: D(A^j) \longrightarrow D(A^{j-1}), j=0, 1, $ is a self-adjoint, non-negative  compact operator, and  following 
additional conditions: 
\begin{itemize}
\item[(i)] There exists $\kappa \in [0,1[$ such that $xf(x) + \kappa x$ is a non-decreasing function on  $[\mu_1, +\infty)$.

\item[(ii)] $AK - KA : X \longrightarrow X$ is a compact operator.

\end{itemize}

\noindent Then, the system \eqref{observability4} is exactly observable.     

 \end{theorem}
 
 \begin{remark} 
 Notice that the two conditions are independent and of different types. Indeed, we will see later that condition $(i)$ ensures that the eigenvalues of $A + K$ satisfy a gap condition, while condition $(ii)$ guarantees that the high-frequency eigenfunctions of the same operator satisfy the inequality \eqref{observabilitycondition}.
  \end{remark}
 
 \section{Proof of Theorem \ref{eigenvalues}}
\begin{proof}
    
 \noindent First, we use the  mini-max Theorem (see \cite{MB}) to characterize the eigenvalues of the two operators $A$ and $A_K$.  

\medskip

\noindent Since $A : D(A) \subseteq X \to X$ and $A + K : D(A) \subseteq X \to X$ are two self-adjoint, non-negative operators with compact resolvent, their eigenvalues admit the following characterizations:
\begin{equation*}
    \mu_n \;=\; \min_{V_n \subseteq D(A)} \ 
    \max_{\substack{\phi \in V_n \\ \|\phi\|_{X} = 1}} 
    \langle A\phi, \phi \rangle_X,
\end{equation*}
and
\begin{equation*}
    \tilde{\mu}_n \;=\; \min_{{V}_n \subseteq D(A)} \ 
    \max_{\substack{\phi \in {V}_n \\ \|\phi\|_{X} = 1}} 
    \langle A_K \phi, \phi \rangle_X,
\end{equation*}
where $V_n$  denote an $n$-dimensional subspace of $D(A)$.

\medskip

\noindent We now distinguish two cases.  Recall that 
 the minimum in the  expressions above, is attained when $V_n$ coincides with the finite-dimensional space  $E_{n}= \textrm{span}\left\{\phi_{k}, k \leq n; \;  A\phi_{k} = \mu_{k}\phi_{k} \right\} $ for $\mu_n$  and $\widetilde E_{n}= \textrm{span}\left \{ \tilde \phi_{k}, k \leq n; \; A_K\tilde \phi_{k} = \tilde \mu_{k} \tilde \phi_{k} \right \}$ for $\tilde \mu_n$.
We next consider two different cases: 
\\
\\
\textbf{First case: $V_n= E_{n}$. } Consequently 

\begin{equation*} 
\tilde{\mu}_n - \mu_n \;\leq\; 
\max_{\phi_n \in E_n; \|\phi\|_{X} = 1} \langle (A + K)\phi_n, \phi_n \rangle_X 
\;-\; 
\max_{\phi_n \in E_n; \|\phi\|_{X} = 1} \langle A\phi_n, \phi_n \rangle_X,
\end{equation*}

\noindent Since ${E}_n$ is a finite-dimensional space,  the first maximum is attained at some vector $\hat{\phi}_n \in E_n$ with $\|\hat{\phi}_n\|_X = 1$.  

\medskip

\noindent Therefore, we obtain
\begin{equation*}
\tilde{\mu}_n - \mu_n 
\;\leq\; 
\langle (A+K)\hat{\phi}_n, \hat{\phi}_n \rangle_X 
- \langle A \hat{\phi}_n, \hat{\phi}_n \rangle_X
= \langle K\hat{\phi}_n, \hat{\phi}_n \rangle_X.
\end{equation*}

\medskip

\noindent Since $A\hat{\phi}_n = \mu_n \hat{\phi}_n$, it follows that
$A^{-1}\hat{\phi}_n = \tfrac{1}{\mu_n}\hat{\phi}_n$.  
Thus, we deduce
\begin{equation*}
\tilde{\mu}_n 
\;\leq\;  
\mu_n \Big( 1 + \langle K A^{-1}\hat{\phi}_n, \hat{\phi}_n \rangle_X \Big).
\end{equation*}

 \noindent \textbf{Second case: $V_n= \widetilde E_{n}$. }

\begin{equation*} 
\mu_n - \tilde{\mu}_n \;\leq\; 
\max_{\tilde{\phi}_n \in \tilde{E}_n} \langle A \tilde{\phi}_n, \tilde{\phi}_n \rangle_X  
\;-\;  
\max_{\tilde{\phi}_n \in \tilde{E}_n} \langle (A + K)\tilde{\phi}_n, \tilde{\phi}_n \rangle_X.
\end{equation*}

\noindent Since $\tilde{E}_n$ is a finite-dimensional space, the first maximum is attained at some vector $\widetilde{\hat{\phi}}_n \in \tilde{E}_n$ with $\|\widetilde{\hat{\phi}}_n\|_X = 1$.  

\medskip

\noindent Thus, we obtain
\begin{equation*}
\mu_n - \tilde{\mu}_n \;\leq\; 
\langle A \widetilde{\hat{\phi}}_n, \widetilde{\hat{\phi}}_n \rangle_X 
- \langle (A+K)\widetilde{\hat{\phi}}_n, \widetilde{\hat{\phi}}_n \rangle_X 
= -\,\langle K \widetilde{\hat{\phi}}_n, \widetilde{\hat{\phi}}_n \rangle_X.
\end{equation*}

\noindent Since $(A+K)\widetilde{\hat{\phi}}_n = \tilde{\mu}_n \widetilde{\hat{\phi}}_n$, it follows that
\[
(A+K)^{-1}\widetilde{\hat{\phi}}_n = \frac{1}{\tilde{\mu}_n}\,\widetilde{\hat{\phi}}_n.
\]
Therefore,
\begin{equation*}
\mu_n \;\leq\; 
\tilde{\mu}_n \Big( 1 - \langle K (A+K)^{-1}\widetilde{\hat{\phi}}_n, \widetilde{\hat{\phi}}_n \rangle_X \Big).
\end{equation*}

\medskip

\noindent Now, set
\[
\alpha_n = \langle K A^{-1}\hat{\phi}_n, \hat{\phi}_n \rangle_X,
\qquad 
\beta_n = \langle K (A+K)^{-1}\widetilde{\hat{\phi}}_n, \widetilde{\hat{\phi}}_n \rangle_X.
\]

Since $A$ is positive and $K$ is non-negative we have $  0 \leq \alpha_n $ and $ 0\leq \beta_n <1$.
\noindent We then obtain
\begin{equation*}
\frac{1}{1 - \beta_n} \;\leq\; \frac{\tilde{\mu}_n}{\mu_n} \;\leq\; 1 + \alpha_n.
\end{equation*}

\medskip

\noindent We next show that $\alpha_n \to 0$ and $\beta_n \to 0$, which yields
\[
\frac{\tilde{\mu}_n}{\mu_n} \;\longrightarrow\; 1, 
\qquad \text{as } n \to +\infty.
\]

\begin{proposition} \label{pop1}
Since $D(A)$ is dense in $X$ and $K A^{-1} \colon X \to X$ is a compact operator, it follows that
\[
\lim_{n \to \infty} \alpha_n = 0.
\]
\end{proposition}

\begin{proof}
Recall  both assumptions are true. Precisely  $A$ is densely defined and  since  $K : D(A) \to X$, is a compact operator,  $KA^{-1} : X \to X$ is also compact. 

\medskip

Since $KA^{-1} : X \to X$ is compact and $(\phi_n)_{n \in \mathbb{N}}$ is uniformly bounded in $X$, there exists a convergent subsequence $(KA^{-1}\phi_n)_{n \in \mathbb{N}}$ such that
\[
KA^{-1}\phi_n \longrightarrow \phi \in X \quad \text{as } n \to +\infty,
\]
for some $\phi \in X$. Next we show that  $\phi =0$.\\

\medskip

\noindent Now, let $\psi \in X$. Since $(\phi_n)_{n \in \mathbb{N}^*}$ is an orthonormal basis of $D(A)$, we have
\[
A^{-1}K\psi = \sum_{n=1}^{+\infty} \langle A^{-1}K\psi, \phi_n \rangle_X \, \phi_n,
\]
and
\[
\|A^{-1}K\psi\|_X^2 = \sum_{n=1}^{+\infty} \langle A^{-1}K\psi, \phi_n \rangle_X^{2} < +\infty.
\]
Therefore
\[
\forall \psi \in X, \qquad \langle A^{-1}K\psi, \phi_n \rangle_X \longrightarrow 0 
\quad \text{as } n \to +\infty.
\]

\medskip

\noindent Since $A^{-1}$ and $K$ are both self-adjoint operators, we have
\[
\langle A^{-1}K\psi, \phi_n \rangle_X 
= \langle \psi, KA^{-1}\phi_n \rangle_X, 
\qquad \forall \psi \in X.
\]
Thus,
\[
\forall \psi \in X, \qquad 
\langle \psi, KA^{-1}\phi_n \rangle_X \longrightarrow 0 \quad \text{as } n \to +\infty.
\]

\medskip

\noindent Consequently, we have
\[
KA^{-1}\phi_n \rightharpoonup 0 \quad \text{as } n \to +\infty.
\]

Hence 
 \[
 \phi = 0.
 \]
 Since all convergent subsequence of $KA^{-1}\phi_n$ converges to $0$, the whole sequence 
 converges to $0$, which finishes the proof.
 
\end{proof}

\begin{proposition} \label{pop2}
Since $D(A)$ is dense in $X$ and $K (A+K)^{-1} \colon X \to X$ is a compact operator, it follows that
     \[\lim_{n \rightarrow \infty} \beta_n = 0.\]
 \end{proposition}
 
\begin{proof}
The proof is similar to the proof of the previous Proposition. Both assumptions of the Proposition are 
satisfied. \\

\medskip

Since $K(A+K)^{-1} : X \to X$ is compact and $(\tilde \phi_n)_{n \in \mathbb{N}}$ is uniformly bounded in $X$, there exists a convergent subsequence $(K(A+K)^{-1}\phi_n)_{n \in \mathbb{N}}$ such that
\[
K(A+K)^{-1}\tilde \phi_n \longrightarrow \phi \in X \quad \text{as } n \to +\infty,
\]
for some $\phi \in X$. Next we show that  $\phi =0$.\\

\medskip

\noindent Now, let $\psi \in X$. Since $(\tilde \phi_n)_{n \in \mathbb{N}^*}$ is an orthonormal basis of $D(A)$, we have
\[
(A+K)^{-1}K\psi = \sum_{n=1}^{+\infty} \langle A^{-1}K\psi, \tilde \phi_n \rangle_X \, \tilde \phi_n,
\]
and
\[
\|(A+K)^{-1}K\psi\|_X^2 = \sum_{n=1}^{+\infty} \langle (A+K)^{-1}K\psi,  \tilde \phi_n \rangle_X^{2} < +\infty.
\]
Therefore
\[
\forall \psi \in X, \qquad \langle (A+K)^{-1}K\psi, \tilde \phi_n \rangle_X \longrightarrow 0 
\quad \text{as } n \to +\infty.
\]

\medskip

\noindent Since $(A+K)^{-1}$ and $K$ are both self-adjoint operators, we have
\[
\langle (A+K)^{-1}K\psi, \tilde \phi_n \rangle_X 
= \langle \psi, K(A+K)^{-1}\tilde \phi_n \rangle_X, 
\qquad \forall \psi \in X.
\]
Thus,
\[
\forall \psi \in X, \qquad 
\langle \psi, K(A+K)^{-1}\phi_n \rangle_X \longrightarrow 0 \quad \text{as } n \to +\infty.
\]

\medskip

\noindent Consequently, we have
\[
K(A+K)^{-1}\phi_n \rightharpoonup 0 \quad \text{as } n \to +\infty,
\]
which implies 
 \[
 \phi = 0.
 \]
 Since all convergent subsequence of $K(A+K)^{-1}\phi_n$ converges to $0$, the whole sequence 
 converges to $0$, which achieves the proof of the Proposition.

\end{proof}

\noindent We deduce from the previous Proposition the following inequalities
\begin{equation*}
     \frac{1}{1-\beta_n} \;\leq\; \frac{\tilde{\mu}_n}{\mu_n} \;\leq\; 1 + \alpha_n.
\end{equation*}
Then, there exists a sequence $(\theta_n)_{n\in \mathbb N^*}$  verifying 
\[
\frac{1}{1 - \beta_n} - 1 \leq \theta_n \leq \alpha_n;
\qquad \lim_{n \to \infty} \theta_n = 0,
\]
and 
\begin{equation*}
    \tilde{\mu}_n = \mu_n \big( 1 + \theta_n \big).
\end{equation*}

\medskip

\noindent By interpolation techniques one can construct a continuous  function 
$f : \mathbb{R}_+ \to \mathbb{R}$ satisfying  

\[\lim_{x \to +\infty} f(x) = 0; \quad f(\mu_n)= \theta_n, \;\;  n\in \mathbb N^*.\]

Finally, we obtain 
\begin{equation*}
    \tilde{\mu}_n = \mu_n \big( 1 + f(\mu_n) \big).
\end{equation*}
\end{proof}

\section{Proof of Theorem \ref{main theorem}}
\noindent We will show that the gap condition of the perturbated eigenvalues is verified under conditions  of 
Theorem \ref{main theorem}. 
\begin{lemma}
The following gap condition:
\begin{equation} \label{gap}
\tilde \mu_{n+1} -\tilde \mu_{n} > \tilde \gamma,
\end{equation}
holds for all $n\in \mathbb N^*$ with $\tilde \gamma = (1 - \kappa)\gamma$.

\end{lemma}

\begin{proof}
Since there exists $\kappa \in [0,1)$ such that the function 
\[
x \mapsto x f(x) + \kappa x,
\]
is non-decreasing, and since the spectral gap condition 
\[
\mu_{n+1} - \mu_{n} > \gamma,
\]
holds for some positive constant $\gamma$, we deduce that
\[
\mu_{n+1}f(\mu_{n+1}) + \kappa \mu_{n+1} 
> \mu_{n} f(\mu_{n}) + \kappa \mu_{n},
\]
that is,
\[
\mu_{n+1}f(\mu_{n+1}) - \mu_{n} f(\mu_{n}) 
> -\kappa (\mu_{n+1} - \mu_{n}).
\]

\medskip

\noindent Therefore, we can write
\begin{align*}
\tilde{\mu}_{n+1} - \tilde{\mu}_{n} 
&= \mu_{n+1}(1 + f(\mu_{n+1})) - \mu_{n}(1 + f(\mu_{n})) \\
&= (\mu_{n+1} - \mu_{n}) + \big( \mu_{n+1}f(\mu_{n+1}) - \mu_{n} f(\mu_{n}) \big) \\
&> (\mu_{n+1} - \mu_{n}) - \kappa (\mu_{n+1} - \mu_{n}) \\
&= (1 - \kappa)(\mu_{n+1} - \mu_{n}) \\
&\geq (1 - \kappa)\gamma.
\end{align*}

\noindent Hence, by taking 
\[
\tilde{\gamma} = (1 - \kappa)\gamma > 0,
\]
we obtain the desired inequality 
\[
\tilde{\mu}_{n+1} - \tilde{\mu}_{n} > \tilde{\gamma}.
\]
\end{proof}

\begin{lemma}
\noindent Assume that $R= AK - KA$ is a compact operator.   For $k \in \mathbb{N}^*$ set  $\tilde{P}_{k}$ the spectral projection onto the eigenspace of $A_{K}$ corresponding to the eigenvalue $\tilde{\mu}_{k}$, that is $\tilde{P}_{k} = \langle \cdot , \tilde \phi_k \rangle_X \tilde \phi_k$. \\

Then
\begin{align*}
   \tilde{P}_{k} K &= K \tilde{P}_{k} 
   + \tilde{F}_j(\tilde{\mu}_k) \, R \tilde{P}_{k} 
   + \tilde{P}_{k} R \, \tilde{F}_j(\tilde{\mu}_k),
\end{align*}
where
\[
\tilde{F}_j(\tilde{\mu}_{k}) 
= \sum_{\substack{j=1 \\ j \neq k}}^{\infty} 
\frac{\tilde{P}_{j}}{\tilde{\mu}_{k} - \tilde{\mu}_{j}}.
\]
\end{lemma}

\begin{proof}
We define the complex disc $B_{\epsilon}$ centered at $\tilde{\mu}_{k}$ with radius 
$\epsilon = \tfrac{\tilde{\gamma}}{4}$.   

\noindent The resolvents $(\mu I - A)^{-1}$ and $(\mu I - (A+K))^{-1}$ are well defined as operators from $X$ onto $D(A) \subset X$, for all
\[
\mu \in \partial B_{\epsilon} := \{ \, \mu \in \mathbb{C} : |\mu - \tilde{\mu}_{k}| = \epsilon \, \}.
\]

\noindent   The gap condition \eqref{gap} implies that $\tilde \mu_k$ is the only eigenvalue of 
$A+K$ within  the complex disc $\overline{B_\epsilon}$.   By the classical Riesz formula, we have
\[
\tilde{P}_{k} = \frac{-1}{2i\pi} \int_{\partial B_{\epsilon}} (\mu I - (A + K))^{-1} \, d\mu,
\]
where $i$ is the imaginary complex number and $I$ denotes the identity operator.

\medskip

\noindent Hence,
\begin{align*}
\tilde{P}_{k} K
&= \frac{1}{2i\pi}\int_{\partial B_\epsilon} (\mu I - (A + K))^{-1} K \, d\mu \\
&= \frac{1}{2i\pi}\int_{\partial B_\epsilon} K (\mu I - (A + K))^{-1} \, d\mu \\
&\quad + \frac{1}{2i\pi}\int_{\partial B_\epsilon} \Big[ (\mu I - (A + K))^{-1} K - K (\mu I - (A + K))^{-1} \Big] \, d\mu.
\end{align*}

\noindent Using the commutator relation $AK = KA + R$, we obtain
\begin{align*}
\tilde{P}_{k} K 
&= K \tilde{P}_{k} + \frac{1}{2i\pi}\int_{\partial B_\epsilon} (\mu I - (A + K))^{-1} R (\mu I - (A+K))^{-1} \, d\mu \\
&= K \tilde{P}_{k} + \frac{1}{2i\pi}\int_{\partial B_\epsilon} 
\Bigg(\frac{\tilde{P}_{k}}{\mu - \tilde{\mu}_{k}} + \tilde{F}_{j}(\mu)\Bigg) 
R \Bigg(\frac{\tilde{P}_{k}}{\mu - \tilde{\mu}_{k}} + \tilde{F}_{j}(\mu)\Bigg) 
\, d\mu,
\end{align*}
where
\[
\tilde{F}_{j}(\mu) = \sum_{\substack{j=1 \\ j \neq k}}^{\infty} \frac{\tilde{P}_{j}}{\mu - \tilde{\mu}_{j}}
\]
is holomorphic for $\mu \in \partial B_\epsilon$.

\medskip

\noindent Moreover, we get
\begin{align*}
\tilde{P}_{k} K
&= K \tilde{P}_{k} 
+ \frac{1}{2i\pi}\int_{\partial B_\epsilon} \frac{\tilde{P}_{k} R \tilde{P}_{k}}{(\mu - \tilde{\mu}_{k})^{2}} \, d\mu 
+ \frac{1}{2i\pi}\int_{\partial B_\epsilon} \frac{\tilde{F}_{j}(\mu) R \tilde{P}_{k}}{\mu - \tilde{\mu}_{k}} \, d\mu \\
&\quad + \frac{1}{2i\pi}\int_{\partial B_\epsilon} \frac{\tilde{P}_{k} R \tilde{F}_{j}(\mu)}{\mu - \tilde{\mu}_{k}} \, d\mu 
+ \frac{1}{2i\pi}\int_{\partial B_\epsilon} \tilde{F}_{j}(\mu) R \tilde{F}_{j}(\mu) \, d\mu.
\end{align*}

\noindent By the Residue theorem
\[
\int_{\partial B_\epsilon} \frac{\tilde{P}_{k} R \tilde{P}_{k}}{(\mu - \tilde{\mu}_{k})^{2}} \, d\mu = 0.
\]
Moreover, since $\tilde{F}_{j}$ is holomorphic on $\partial B_\epsilon$ and $R$ is bounded, we also have
\[
\int_{\partial B_\epsilon} \tilde{F}_{j}(\mu) R \tilde{F}_{j}(\mu) \, d\mu = 0.
\]

\medskip

\noindent Therefore
\begin{align*}
\tilde{P}_{k} K
&= K \tilde{P}_{k} 
+ \frac{1}{2i\pi}\int_{\partial B_\epsilon} \frac{\tilde{F}_{j}(\mu)}{\mu - \tilde{\mu}_{k}} \, d\mu \; R \tilde{P}_{k} 
+ \tilde{P}_{k} R \; \frac{1}{2i\pi}\int_{\partial B_\epsilon} \frac{\tilde{F}_{j}(\mu)}{\mu - \tilde{\mu}_{k}} \, d\mu \\
&= K \tilde{P}_{k} + \tilde{F}_{j}(\tilde{\mu}_{k}) R \tilde{P}_{k} + \tilde{P}_{k} R \tilde{F}_{j}(\tilde{\mu}_{k}),
\end{align*}
with
\[
\tilde{F}_{j}(\tilde{\mu}_{k}) = \sum_{\substack{j=1 \\ j \neq k}}^{\infty} \frac{\tilde{P}_{j}}{\tilde{\mu}_{k} - \tilde{\mu}_{j}}.
\]
\end{proof}

\begin{proof}[Proof of Theorem \ref{main theorem}] 
 First we observe  that the operator 
\[
\tilde{P}_k K : \ker(\tilde{\mu}_{k}I - A_{K}) \;\longrightarrow\; \ker(\tilde{\mu}_{k}I - A_{K}),
\] 
is self-adjoint, non-negative.\\

 \noindent Let  $\sigma_k \geq 0$ be the eigenvalues and $\tilde{\psi}_k$ be  the corresponding normalized eigenfunctions of the operator $\tilde{P}_k K$, that is 
\[
\tilde{P}_k K \tilde{\psi}_{k} = \sigma_k \tilde{\psi}_{k}, \;\; \| \tilde \psi_k\|_X =1.
\]

\noindent  Remark that $ \tilde{\psi}_k \in \ker(\tilde{\mu}_{k}I - A_{K})$, and so it is also an eigenfunction of 
the operator $A+K$. 

\medskip

\noindent Since
\begin{equation*}\label{Hautus condition}
 \|(A - \omega I)z\|_{X}^{2} + \|Cz\|_{Y}^{2} \;\geq\; \rho \|z\|_{X}^{2},  
 \qquad \forall \omega \in \mathbb{R}, \quad \forall z \in D(A),
\end{equation*}
and
\[
(A+K)\tilde{\psi}_{k} = \tilde{\mu}_{k} \tilde{\psi}_{k},
\]
we obtain
\[
\| ((\tilde{\mu}_{k} - \omega)I - K)\tilde{\psi}_{k} \|_{X}^{2} 
+ \|C\tilde{\psi}_{k}\|_{Y}^{2} \;\geq\; \rho.
\]

\noindent Setting now $\omega = \tilde{\mu}_{k} -  \sigma_{k}$,  we  get
\[
\| (\sigma_{k}I - \tilde{P}_k K + \tilde{P}_k K - K)\tilde{\psi}_{k} \|_{X}^{2} 
+ \|C\tilde{\psi}_{k}\|_{Y}^{2} \;\geq\; \rho,
\]
that is
\begin{equation} \label{wahed}
\|(I - \tilde{P}_k)K \tilde{\psi}_{k}\|_{X}^{2} + \|C\tilde{\psi}_{k}\|_{Y}^{2} \;\geq\; \rho.
\end{equation}

\medskip

\noindent Moreover
\begin{align*}
\|(I - \tilde{P}_k)K \tilde{\psi}_{k}\|_{X}^{2} 
&= \|K\tilde{\psi}_{k} - \tilde{P}_k K \tilde{\psi}_{k}\|_{X}^{2} \\
&= \|K\tilde{\psi}_{k} - K \tilde{P}_{k}\tilde{\psi}_{k} 
- \tilde{F}_j(\tilde{\mu}_k) R \tilde{P}_{k}\tilde{\psi}_{k} 
- \tilde{P}_{k} R \tilde{F}_j(\tilde{\mu}_k)\tilde{\psi}_{k}\|_{X}^{2} \\
&= \|\tilde{F}_j(\tilde{\mu}_k) R \tilde{P}_{k}\tilde{\psi}_{k} 
+ \tilde{P}_{k} R \tilde{F}_j(\tilde{\mu}_k)\tilde{\psi}_{k}\|_{X}^{2} \\
&= \|\tilde{F}_j(\tilde{\mu}_k) R \tilde{P}_{k}\tilde{\psi}_{k}\|_{X}^{2} \\
&\leq \frac{4}{\tilde{\gamma}} \|R \tilde{\psi}_{k}\|_{X}^{2}.
\end{align*}

\noindent Since $R : X \to X$ is compact and $(\tilde{\psi}_{k})_{k \in \mathbb{N}}$ is uniformly bounded, similar 
arguments in the proofs of Propositions \ref{pop1} and \ref{pop2} lead to the strong convergence 
of  the sequence $(R\tilde{\psi}_{k})_{k \in \mathbb{N}}$  to zero, that is 
\[
\|R\tilde{\psi}_{k}\|_X \rightarrow 0, \quad k\rightarrow +\infty.
\]

\noindent Finally, we get 
\begin{align*}   
\left \lVert ((I -\tilde P_k)K \tilde \psi_{k} \right \rVert^{2}_{X} &\leq \frac{4}{\tilde \gamma} \left \lVert R \tilde \psi_{k} \right \rVert^{2}_{X} \rightarrow 0, \quad  k \rightarrow +\infty.
 \end{align*}

\noindent Therefore there exists $k_\rho>0$ such that 
\begin{equation} \label{ethnan}
\left \lVert ((I -\tilde P_k)K \tilde \psi_{k} \right \rVert^{2}_{X} \leq  \frac{4}{\tilde{\gamma}} \|R \tilde{\psi}_{k}\|_{X}^{2} \leq \frac{\rho}{2}. \qquad \forall k > k_\rho.
\end{equation}
\noindent  Combining inequalities \eqref{wahed} and \eqref{ethnan}, we obtain 

\[
\|C \tilde{\psi}_k\|^2_Y \;\geq\; \frac{\rho}{2},
\qquad \forall k > k_\rho.
\]
\medskip

\noindent  On the other hand, we deduce from assumption \eqref{necessarycondition}  

\[c_{k_\rho} = \min_{k \leq k_\rho} \| C \tilde{\phi}_k \|_{Y} >0.\] 

Taking  $\tilde \delta = \min( c_{k_\rho}, \frac{\rho}{2}),$  we finally obtain 

\[
\|C \tilde{\phi}_k\|^2_Y \;\geq\; \tilde{\delta}, \qquad \forall k\in \mathbb N^*.
\]

\noindent Now, with the gap condition \eqref{gap} in mind,  and  according to the spectral observability criterion in Theorem \ref{thm1}, since $\tilde{\delta} > 0$  the perturbed system \eqref{observability4} is exactly observable, which finishes the proof. 
\end{proof}

\end{document}